\documentclass[12pt]{article}
\usepackage{booktabs}
\usepackage{caption}
\usepackage{mathrsfs}
\usepackage{amsmath}
\usepackage{amsfonts,amsthm,amssymb,mathrsfs,bbding}
\usepackage{graphics,multicol}
\usepackage{graphicx}
\usepackage{color}
\usepackage{enumerate}
\usepackage{caption}
\usepackage{rotating}
\usepackage{lscape}
\usepackage{longtable}
\allowdisplaybreaks[4]
\usepackage{tabularx}
\usepackage[colorlinks=true,anchorcolor=blue,filecolor=blue,linkcolor=blue,urlcolor=blue,citecolor=blue]{hyperref}
\usepackage{extarrows}
\usepackage{cite}
\usepackage{latexsym,bm}
\usepackage{mathtools}
\evensidemargin=\oddsidemargin\topmargin=-1.5cm

\newtheorem{thm}{Theorem}

\newtheorem{lem}[thm]{Lemma}
\newtheorem{cor}[thm]{Corollary}

\newtheorem{fact}{Fact}
\newtheorem{remark}{Remark}

\newtheorem{claim}{Claim}
\theoremstyle{definition}
\renewcommand\proofname{\it Proof}
\addtocounter{section}{0}
\begin{document}

\title{\bf The signless Laplacian spectral radius of graphs with no intersecting triangles}
\author{{Yanhua Zhao, Xueyi Huang\footnote{Corresponding author. }\setcounter{footnote}{-1}\footnote{\emph{Email address:} huangxymath@163.com.} \ and Hangtian Guo}\\[2mm]
\small Department of Mathematics, East China University of Science and Technology, \\
\small  Shanghai 200237, P.R. China}


\date{}
\maketitle
{\flushleft\large\bf Abstract}  Let $F_k$ denote the $k$-fan consisting of $k$ triangles which intersect in exactly one common vertex, and $S_{n,k}$ the complete split graph of order $n$ consisting of a clique on  $k$ vertices and an independent set on the remaining vertices in which each vertex of the clique is adjacent to each vertex of the independent set.  In this paper, it is shown that $S_{n,k}$ is the unique  graph attaining the maximum signless Laplacian spectral radius among all graphs of order $n$ containing no $F_k$, provided that $k\geq 2$ and $n\geq 3k^2-k-2$.
\begin{flushleft}
\end{flushleft}

\section{Introduction}
All graphs considered here are simple and undirected. Let $G $ be a graph with vertex set $V(G)$,  edge set $E(G)$ ($e(G)=|E(G)|$) and adjacency matrix $A(G)$. For any $v\in V(G)$, we denote by $N_k(v)$  the set of vertices at distance $k$ from $v$ in $G$. In particular,  $N(v)=N_1(v)$ and $d_v=|N(v)|$ are  the neighborhood and degree of $v$, respectively. The \textit{signless Laplacian matrix} of $G$ is defined as $Q(G)=D(G)+A(G)$, where $D(G)=\mathrm{diag}(d_v:v\in V(G))$. The largest eigenvalue of  $Q(G)$ (resp. $A(G)$) is called the \textit{signless Laplacian spectral radius} (resp. \textit{spectral radius}) of $G$, and denoted by $q_1(G)$.  For some interesting properties and bounds of $q_1(G)$, we refer the reader to \cite{CS,CS2,CS3,NP}.  If $M$ is a square matrix of order $n$ with only real eigenvalues, we arrange its eigenvalues in  non-increasing order $\lambda_1(M)\geq \lambda_2(M)\geq \cdots \geq \lambda_n(M)$.

For $S,T\subseteq V(G)$ with $S\cap T=\emptyset$, let $e(S,T)$ be the number of edges between $S$ and  $T$, and $G[S]$ the subgraph of $G$ induced by $S$.  As usual, we denote by  $K_n$ the complete graph on $n$ vertices, $kG$  the disjoint union of $k$ copies of $G$, and  $G\nabla H$  the graph obtained from the disjoint union $G\cup H$ by adding all edges between $G$ and $H$. In particular,  $F_k=K_1\nabla kK_2$ and  $S_{n,k}=K_k\nabla (n-k)K_1$.

Given a graph $H$, we say that $G$ is  $H$-\textit{free} if it does not contain $H$ as a subgraph.   The \textit{Tur\'{a}n number} of $H$, denoted by $ex(n,H)$, is the maximum number of edges in an $H$-free graph of order $n$.  Let $Ex(n,H)$ denote the set of $H$-free graphs of order $n$ with $ex(n,H)$ edges. To determine $ex(n,H)$ and characterize those graphs in $Ex(n,H)$ is  a fundamental problem (called Tur\'{a}n-type problem) in extremal graph theory, and the reader is referred to \cite{CF,NI5,SI} for surveys on this topic. In 1995, Erd\H{o}s, F\"{u}redi, Gould and Gunderson
\cite{EFGG} considered the Tur\'{a}n-type problem for $F_k$-free graphs, and they established  the following result.  

\begin{thm}(Erd\H{o}s, F\"{u}redi, Gould and Gunderson, \cite{EFGG})\label{thm-1}
 For $k \geq 1$ and $n \geq 50k^2$, we have
$$
ex(n,F_k)=\Big\lfloor\frac{n^2}{4}\Big\rfloor+\left\{
  \begin{array}{ll}
    k^2-k, & \mbox{if $k$ is odd;} \\
    k^2-\frac{3}{2}k, & \mbox{if $k$ is even.}
  \end{array}
\right.
$$
Furthermore, the number of edges is best possible.
If $k$ ($n \geq 4k-1$) is odd, then the unique extremal graph is constructed by taking a complete equi-bipartite graph  and embedding two vertex disjoint copies of $K_k$ in one side; if $k$ ($n \geq 4k-3$) is even, then the extremal graph is constructed by taking a complete equi-bipartite graph and
embedding a graph with $2k-1 $ vertices, $ k^2-\frac{3k}{2}$ edges and maximum degree $k-1 $ in one side. 
\end{thm}

In extremal spectral graph theory, the Brualdi-Solheid-Tur\'{a}n type problem proposed by Nikiforov \cite{NI3}  asks for the maximum spectral radius of an $H$-free graph of order $n$. Up to now, this problem has been studied for various kinds of $H$ such as the complete graph \cite{WI}, the complete bipartite graph \cite{BG,NI4},  and the cycles or paths of specified  length \cite{NI2,NI3,ZW,ZL,GH}. 
For the signless Laplacian spectral radius, the Brualdi-Solheid-Tur\'{a}n type problem has also been investigated for those graph classes mentioned above \cite{FNP,NY,NY1,YU}.

Very recently, Cioab\u{a}, Feng, Tait and Zhang \cite{CFTZ} studied the Brualdi-Solheid-Tur\'{a}n type problem for graphs contatining no $F_k$, and gave the following result.

\begin{thm}(Cioab\u{a}, Feng, Tait and Zhang \cite{CFTZ})\label{thm-2}
Let $G$ be a $F_k$-free graph of order $n$. For sufficiently large $n$, if $G$ has the maximal spectral radius, then
$$G \in Ex(n, F_k),$$
where $Ex(n, F_k)$ consists of the extremal graphs given in Theorem \ref{thm-1}.
\end{thm}

Inspired by the work of Cioab\u{a}, Feng, Tait and Zhang \cite{CFTZ}, in this paper, we focus on the maximum signless Laplacian spectral radius of $F_k$-free graphs, and prove that

\begin{thm}\label{thm-3}
Let $k\geq 2$ and $n\geq 3k^2-k-2$. If $G$ is a $F_k$-free graph of order $n$, then
$$q_1(G) \leq q_1(S_{n,k}),$$
with equality holding if and only if $G=S_{n,k}$.
\end{thm}
\begin{remark}
\emph{
It is worth mentioning that the extremal graphs in Theorem \ref{thm-3} are not the same as those  of Theorem \ref{thm-2}.  In addition, for $k=1$, i.e., $G$ is triangle-free,  from \cite[Theorem 1.3]{HJZ} we know that $q_1(G)\leq n$ with equality holding if and only if $G$ is a complete bipartite graph. 
}
\end{remark}

\section{Proof of Theorem \ref{thm-3}}

First of all, we list some  lemmas, which are crucial  for the proof of Theorem \ref{thm-3}.

Let $M$ be a real $n\times n$ matrix, and let $\mathcal{N}=\{1,2,\ldots,n\}$. Given a partition $\Pi:\mathcal{N}=\mathcal{N}_1\cup \mathcal{N}_2\cup \cdots \cup \mathcal{N}_k$,  the matrix $M$ can be correspondingly partitioned as
$$
M=\left(\begin{array}{ccccccc}
M_{1,1}&M_{1,2}&\cdots &M_{1,k}\\
M_{2,1}&M_{2,2}&\cdots &M_{2,k}\\
\vdots& \vdots& \ddots& \vdots\\
M_{k,1}&M_{k,2}&\cdots &M_{k,k}\\
\end{array}\right).
$$
The \textit{quotient matrix} of $M$ with respect to $\Pi$ is defined as the $k\times k$ matrix $B_\Pi=(b_{i,j})_{i,j=1}^k$ where $b_{i,j}$ is the  average value of all row sums of  $M_{i,j}$.
The partition $\Pi$ is called \textit{equitable} if each block $M_{i,j}$ of $M$ has constant row sum $b_{i,j}$.
Also, we say that the quotient matrix $B_\Pi$ is \textit{equitable} if $\Pi$ is an  equitable partition of $M$.

\begin{lem}(Brouwer and Haemers \cite[p. 30]{BH}; Godsil and Royle \cite[pp. 196--198]{GR})\label{lem-1}
Let $M$ be a real symmetric matrix, and let $B_\Pi$ be an equitable quotient matrix of $M$. Then the eigenvalues of  $B_\Pi$ are also eigenvalues of $M$. Furthermore, if $M$ is nonnegative and irreducible, then 
$$\lambda_1(M) = \lambda_1(B_\Pi).$$
\end{lem}

\begin{lem}(Horn and Johnson \cite[Corollary 8.1.19]{HJ})\label{lem-2}
If $M_1$ and $M_2$ are two nonnegative symmetric  matrices such that  $M_1-M_2$ is nonnegative, then
$$\lambda_1(M_1)\geq \lambda_2(M_2).$$
\end{lem}

The following  bound of $\rho_ Q(G)$ can be traced back to Merris \cite{Me}.

\begin{lem}\label{lem-3}
For every graph $G$, we have $$\rho_ Q(G)\leq \max\Big\{d_v+ \frac{1}{d_v}\sum\limits_{w\in N(v)}d_w: v\in V(G)\Big\}.$$
If $G$ is connected, equality holds if and only if $G$ is regular or semiregular bipartite.
\end{lem}

Let $\alpha'(G)$ denote the matching number of $G$. The following lemma provides the maximum number of edges in a graph of order $n$ with given matching number.

\begin{lem}(Bollob\'{a}s \cite [Corollary 1.10]{BB})\label{lem-4}
If $n\geq 2\alpha+1$, then the maximum size of a graph $G$ of order $n$ with $\alpha'(G)=\alpha$ is
$$
\max\left\{\binom{2\alpha+1}{2}, \alpha n-\frac{(\alpha+1)\alpha}{2}\right\}.
$$
If  $n>(5\alpha+3)/2$ then $S_{n,\alpha}$ is the unique extremal graph; if  $n=(5\alpha+3)/2$ then there are two extremal graphs $K_{2\alpha+1}\cup (n-2\alpha-1)K_1$ and $S_{n,\alpha}$;  if $2\alpha+1\leq n <(5\alpha+3)/2$ then  $K_{2\alpha+1}\cup (n-2\alpha-1)K_1$ is the unique extremal graph.
\end{lem}
If $G$ is $kK_2$-free, then $\alpha'(G)\leq k-1$. According to  Lemma \ref{lem-4}, we obtain the following result immediately.
\begin{cor}\label{cor-1}
Let $k\geq 2$ be an integer. Then
$$
ex(n,kK_2)=\left\{
\begin{array}{ll}
(k-1)n-\frac{k(k-1)}{2},&\mbox{for}~n\geq (5k-2)/2;\\
\binom{2k-1}{2},&\mbox{for}~ 2k-1\leq n<(5k-2)/2.\\
\end{array}
\right.
$$
If   $n>(5k-2)/2$ then $Ex(n,kK_2)=\{S_{n,k-1}\}$; if  $n=(5k-2)/2$ then $Ex(n,kK_2)=\{K_{2k-1}\cup (n-2k+1)K_1, S_{n,k-1}\}$;  if $2k-1\leq n<(5k-2)/2$ then  $Ex(n,kK_2)=\{K_{2k-1}\cup (n-2k+1)K_1\}$.
\end{cor}

\begin{lem}\label{lem-5}
For $n>k\geq 1$, we have
$$q_1(S_{n,k})=\frac{n+2k-2+\sqrt{(n+2k-2)^2-8k(k-1)}}{2}.$$
In particular, if $n\geq 2k^2-4k+3$, then
$$
 q_1(S_{n,k})  \geq n+2k-2-\frac{2k(k-1)}{n+2k-3}.
$$
\end{lem}
\begin{proof}
As $S_{n,k}=K_k\nabla(n-k)K_1$, we see that $Q(S_{n,k})$ has the equitable quotient matrix
$$
B_\Pi=\begin{pmatrix}
n+k-2 & n-k\\
k& k\\
\end{pmatrix}.
$$
By Lemma \ref{lem-1}, we have
$$q_1(S_{n,k})=\lambda_1(B_\Pi)=\frac{n+2k-2+\sqrt{(n+2k-2)^2-8k(k-1)}}{2},$$
as required.  By a simple calculation, the second part of the lemma follows immediately.
\end{proof}

Now we are in a position to give  the proof of Theorem \ref{thm-3}.

\renewcommand\proofname{\it Proof of Theorem \ref{thm-3}}
\begin{proof}
Assume that $G$ has the maximum signless Laplacian spectral radius among all $F_k$-free ($k\geq 2$) graphs of order $n$ ($n\geq 3k^2-k-2$). We claim that $G$ is connected, since otherwise we can add some new edges into $G$ so that the obtained graph $G'$ is connected and still $F_k$-free. However, the Rayleigh quotient and the Perron-Frobenius theorem implies that $q_1(G')>q_1(G)$, contrary to the maximality of $q_1(G)$.  Considering  that $G$ is a connected  $F_k$-free graph of order $n$ with the maximum signless Laplacian spectral radius, we obtain the following two facts.

\begin{fact}\label{fact-1}
For any $v \in V(G)$, $G[N(v)]$ is $kK_2$-free.
\end{fact}

\begin{fact}\label{fact-2}
For any $uv\notin E(G)$,  $G+uv$ contains $F_k$ as a subgraph. Therefore, any two non-adjacent vertices of $G$ have at least one common neighbor, i.e., $V(G)=\{v\}\cup N(v)\cup N_2(v)$ (or $|N_2(v)|=n-1-d_v$) for each $v\in V(G)$.
\end{fact}

Let $u\in V(G)$ be such that
$$d_u+ \frac{1}{d_u}\sum_{w\in N(u)}d_w=\max\Big\{d_v+ \frac{1}{d_v}\sum\limits_{w\in N(v)}d_w: v\in V(G)\Big\}.$$
Notice that $S_{n,k}=K_k\nabla(n-k)K_1$ is $F_k$-free. By Lemma \ref{lem-3} and Lemma \ref{lem-5}, we get
\begin{equation}\label{eq-1}
 \begin{aligned}
  n+2k-2-\frac{2k(k-1)}{n+2k-3}&\leq \rho_ Q(S_{n,k})\\
  &\leq q_1(G)\\
  &\leq d_u+ \frac{1}{d_u}\sum_{w\in N(v)}d_w\\
  &= d_u+ \frac{1}{d_u}[d_u+2e(G[N(u)])+e(N(u),N_2(u))].
  \end{aligned}
\end{equation}

We have the following two claims.

\begin{claim}\label{claim-1}
$d_u\geq (n+2k-3)/2$ and $2e(G[N(u)])+e(N(u),N_2(u))\in [2(k-1)d_u-2k(k-1)+d_u(n-1-d_u)+1,2(k-1)d_u-k(k-1)+d_u(n-1-d_u)]$.
\end{claim}
\renewcommand\proofname{\it Proof}
\begin{proof}
If $d_u\leq 2k-2$, then
$$2e(G[N(u)])+e(N(u),N_2(u))\leq 2\binom{d_u}{2}+d_u(n-1-d_u)=(n-2)d_u.$$
Since $n\geq 3k^2-k-2$ and $k\geq 2$, from the above inequality and (\ref{eq-1})  we can deduce that 
$$
d_u\geq 2k-1-\frac{2k(k-1)}{n+2k-3}>2k-2,
$$
a contradiction.

If $d_u= 2k-1$,  we assert that there are at least $n-\lfloor10k/3\rfloor$ vertices in $N_2(u)$  adjacent to all vertices of $N(u)$, since otherwise we have 
\begin{equation*}
\begin{aligned}
         2e(G[N(u)])+e(N(u),N_2(u))&\leq 2\binom{d_u}{2}+d_u(n-1-d_u)\\
         &~~~-\Big[n-1-d_u-\Big(n-\Big\lfloor\frac{10}{3}k\Big\rfloor-1\Big)\Big] \\
          &=(n-2)(2k-1)-\Big\lfloor\frac{4}{3}k\Big\rfloor-1\\
          &\leq (n-2)(2k-1)-\frac{4}{3}k
 \end{aligned}
 \end{equation*}
which leads to 
$$
n\leq 3k^2-\frac{13}{2}k+\frac{9}{2}
$$ 
by (\ref{eq-1}), contrary to $n\geq 3k^2-k-2$. Notice that $n-\lfloor10k/3\rfloor\geq k-1$ due to $n\geq 3k^2-k-2$ and $k\geq 2$. Thus there are at least $k$ vertices in $\{u\}\cup N_2(u)$ adjacent to all vertices of $N(u)$.  Since $G$ is $F_k$-free, we see that each vertex of $N(u)$ has degree at most $k-1$ in $G[N(u)]$, which gives that
$$
2e(G[N(u)])+e(N(u),N_2(u))\leq (k-1)d_u+d_u(n-1-d_u). 
$$
Combining this with  (\ref{eq-1}) yields that $n\leq 3$, which is impossible because $n\geq 3k^2-k-2\geq 8$. 

If $2k\leq d_u<(5k-2)/2$, by Fact \ref{fact-1} and  Corollary \ref{cor-1}, we obtain
$$
2e(G[N(u)])+e(N(u),N_2(u)) \leq 2\binom{2k-1}{2}+d_u(n-1-d_u),
$$
which gives that 
$$
d_u\leq 2k-1+\frac{k(2k-1)}{n+k-3}<2k
$$
by (\ref{eq-1}) and the fact that $n\geq 3k^2-k-2$, a contradiction.

If $d_u\geq (5k-2)/2$, again by Fact \ref{fact-1} and  Corollary \ref{cor-1}, we have 
\begin{equation*}
2e(G[N(u)])+e(N(u),N_2(u)) \leq 2(k-1)d_u-k(k-1)+d_u(n-1-d_u).
\end{equation*}
Again by  (\ref{eq-1}),  we obtain  $d_u\geq (n+2k-3)/2$, as required. Furthermore, if 
$2e(G[N(u)])+e(N(u),N_2(u)) \leq 2(k-1)d_u-2k(k-1)+d_u(n-1-d_u)$,
then we can deduce that
$d_u\geq n+2k-3$, contrary to $d_u\leq n-1$. Therefore, we conclude that 
$2e(G[N(u)])+e(N(u),N_2(u))\in [2(k-1)d_u-2k(k-1)+d_u(n-1-d_u)+1, 2(k-1)d_u-k(k-1)+d_u(n-1-d_u)]$.

This proves Claim \ref{claim-1}.
\end{proof}

\begin{claim}\label{claim-2}
$G[N(u)]$ is a spanning subgraph of $S_{d_u,k-1}$.
\end{claim}
\begin{proof}
First we assert that $G[N(u)]$ contains  a $(k-1)$-matching $M$, since otherwise we have 
$$
2e(G[N(u)])+e(N(u),N_2(u))\leq 2(k-2)d_u-(k-1)(k-2)+d_u(n-1-d_u)
$$
by Corollary \ref{cor-1} (notice that $d_u\geq (n+2k-3)/2> (5(k-1)-2)/2$), which contradicts Claim \ref{claim-1}. Denote by $M=\{r_is_i:1\leq i\leq k-1\}$. Let $R=\{r_1,\ldots,r_{k-1}\}$, $S=\{s_1,\ldots,s_{k-1}\}$ and $T=N(u)\setminus (R\cup S)$. Since  $G[N(u)]$ is $kK_2$-free by Fact \ref{fact-1}, we see that $T$ must be an independent set. For the same reason, we claim that for each $1\leq i\leq k-1$, at least one of $N(r_i)\cap T$ and $N(s_i)\cap T$ is empty, or $N(r_i)\cap T=N(s_i)\cap T=\{t_i\}$ for some $t_i\in T$. 
Let $p$ denote the number of $i\in[k-1]$ such that $N(r_i)\cap T=\emptyset$ or $N(s_i)\cap T=\emptyset$. If $p\leq k-2$, then 
$$
\begin{aligned}
e(G[N(u)])&\leq \binom{2(k-1)}{2}+(d_u-2(k-1))p+2(k-1-p)\\
&= (d_u-2k)p+(k-1)(2k-1)\\
&\leq  (d_u-2k)(k-2)+(k-1)(2k-1)\\
&= (k-2)d_u +k+1,
\end{aligned}
$$
and therefore,
$$
2e(G[N(u)])+e(N(u),N_2(u))\leq 2(k-2)d_u+2(k+1)+d_u(n-1-d_u).
$$
Combining this with Claim \ref{claim-1}, we obtain $d_u\leq k^2$, which is impossible because $d_u\geq (n+2k-3)/2>k^2$ due to $n\geq 3k^2-k-2$ and $k\geq 2$. Thus we must have $p=k-1$, that is, 
$N(r_i)\cap T=\emptyset$ or $N(s_i)\cap T=\emptyset$ for each $1\leq i\leq k-1$. Without loss of generality, we may assume that $N(s_i)\cap T=\emptyset$ for all $1\leq i\leq k-1$, or equivalently, there are no edges between $S$ and $T$. In what follows, we shall see that  $S$ is also an independent set. In fact, if there exists some edge $s_is_j$ ($i\neq j$) in $G[S]$, as above, we see that
at least one of $N(r_i)\cap T$ and $N(r_j)\cap T$ is empty, or $N(r_i)\cap T=N(r_j)\cap T=\{t\}$ for some $t\in T$. This implies that $e(\{r_i,r_j\},T)\leq \max\{2,d_u-2(k-1)\}=d_u-2(k-1)$, and so
$$
\begin{aligned}
e(G[N(u)])&\leq \binom{2(k-1)}{2}+d_u-2(k-1)+(d_u-2(k-1))(k-3)\\
&= (k-2)d_u +k-1,
\end{aligned}
$$
which is  impossible by above arguments. Concluding these results, we obtain that $S\cup T$ is an independent set. Since $|R|=|S|=k-1$ and $|T|=d_u-2(k-1)$, we see that $G[N(u)]$ is exactly a spanning subgraph of $S_{d_u,k-1}$.
\end{proof}

Let  $R=\{r_1,\ldots,r_{k-1}\}$, $S=\{s_1,\ldots,s_{k-1}\}$ and $T$ are defined as in Claim \ref{claim-2}. We have $N(u)=R\cup S\cup T$. Let $X=\{u\}\cup R\cup N_2(u)$ and $Y=S\cup T$. Notice that $Y$ is an independent set. We assert that there are at most $k(k-1)-1$ vertices of $Y$  not adjacent to all vertices of $X\setminus\{u\}$, since otherwise we can deduce from Claim \ref{claim-2} that  
$$
\begin{aligned}
2e(G[N(u)])+e(N(u),N_2(u))&\leq 2e(S_{d_u,k-1})+d_u(n-1-d_u)-k(k-1)\\
&=2(k-1)d_u-2k(k-1)+d_u(n-1-d_u),
\end{aligned}
$$
which is impossible by Claim \ref{claim-1}. As $Y\subseteq N(u)$ and $|Y|=d_u-k+1$, the number of vertices in $Y$ that are adjacent to all vertices of $X$ is at least 
$$|Y|-(k(k-1)-1)=d_u-k+1-(k(k-1)-1)=d_u-k^2+2\geq k,$$
where the last inequality follows from $d_u\geq (n+2k-3)/2\geq (3k^2+k-5)/2$ and $k\geq 2$. Since $G$ is $F_k$-free, we may conclude that each vertex of $X$ has degree at most $k-1$ in $G[X]$.   Let $G^*=G[X]\nabla G[Y]=G[X]\nabla(d_u-k+1)K_1$. Then $G$ is a spanning subgraph of $G^*$. Take 
$$
Q^*=Q(G^*)+
\begin{bmatrix}
\mathrm{diag}(2(k-1-d_x^*):x\in X) & 0\\
0 & 0\\
\end{bmatrix}
\begin{matrix}
X\\
Y
\end{matrix}
,
$$
where $d_x^*$ denotes the degree of $x\in X$ in $G[X]$. Observe that $Q^*$ has the equitable quotient matrix 
$$
B_\Pi=
\begin{bmatrix}
d_u+k-1 & d_u-k+1\\
n-d_u+k-1 & n-d_u+k-1\\
\end{bmatrix}
\begin{matrix}
X\\
Y
\end{matrix}.
$$
Then, by Lemma \ref{lem-1} and Lemma \ref{lem-2}, we obtain 
\begin{equation}\label{eq-3}
\begin{aligned}
q_1(G)&\leq q_1(G^*)\leq \lambda_1(Q^*)=\lambda_1(B_\Pi)\\
&=\frac{n+2k-2+\sqrt{(n+2k-2)^2+8(k-1)(d_u-n-k+1)}}{2}.
\end{aligned}
\end{equation}
On the other hand, from Lemma \ref{lem-5} we have 
\begin{equation}\label{eq-4}
q_1(G)\geq q_1(S_{n,k})=\frac{n+2k-2+\sqrt{(n+2k-2)^2-8k(k-1)}}{2}.
\end{equation}
Combining (\ref{eq-3}) and (\ref{eq-4}), we  can deduce that $d_u\geq n-1$. Then we have  $d_u=n-1$, and so $G$ is a spanning subgraph of $K_1\nabla S_{n-1,k}=S_{n,k}$ by Claim \ref{claim-2}. Therefore, we must have $G=S_{n,k}$ by the maximality of $q_1(G)$. 

We complete the proof.
\end{proof}

\section*{Acknowledgements}
X. Huang is partially supported by the National Natural Science Foundation of China (Grant No. 11901540 and Grant No. 11671344).

\end{document}